%% file: main_article.tex
\documentclass[hidelinks,onefignum,onetabnum]{siamart251216}

\input{ex_shared}

\ifpdf
\hypersetup{
  pdftitle={Effective Resistances and Commute Times in Random Geometric Graphs},
  pdfauthor={Peter J. Mucha, Filip O. Rupchin}
}
\fi

\begin{document}

\maketitle

\begin{abstract}
The commute time between two nodes in a network --- the expected number of steps for a random walk to travel from one node to the other and then return --- is a metric of broad importance arising in community detection, network routing, dimensionality reduction, and diffusion modeling. For random geometric graphs (RGGs), in which nodes are placed at random in a spatial domain and connected pairwise wherever their Euclidean distance is below a threshold radius, the relationship between commute times and the embedding geometry remains poorly understood outside very dense settings (where the role of the geometry disappears and commute times degenerate to a sum of inverse degrees). We develop and numerically validate a model for approximating commute times in sparse RGGs on a torus by combining theoretically motivated geometric contributions with an inverse degree sum. The geometric terms include a universal logarithmic contribution from the Laplacian, a quadratic correction encoding the compact topology of the torus, and a quartic angular term reflecting the square anisotropy of the domain. We fit this model to samples of node pairs across a range of graph sizes and mean degrees, demonstrating good predictive performance and that the geometric terms contribute significantly to model fit. We then study the continuous perturbation of the model from a regular square lattice to a fully random geometric graph, further validating the functional model form through this transition and showing how commute times in sparse RGGs retain meaningful geometric information about the embedding space.
\end{abstract}



\section{Introduction}
Networks in the real world often obey local spatial constraints. Neurons in the brain are connected by synapses if they are spatially close to one another \cite{penrose2003random}. In communication networks, stations are connected to geographically close neighbors \cite{penrose2003random}. The spread of disease relies on local contact. Friendship networks are strongly influenced by locality (propinquity) and homophily, with similarities between individuals in terms of proximity in some feature space (see, e.g., \cite{McPherson1983}). However, despite their strong dependence on local structure, global properties must also be considered and many of these networks are very large: for example, Instagram reported three billion monthly users in 2025 \cite{cnbc2025instagram}. For such networks, computing global properties exactly is often computationally prohibitive, raising fundamental questions about the extent to which global network properties can be understood or approximated.

Understanding the relationship between local and global structure in spatially embedded networks also speaks to issues about how geometry shapes dynamics on networks. When global properties such as commute times --- which govern how information, disease, or influence spreads across a network --- can be predicted from local and geometric information alone, it suggests how the large-scale behavior of the network is encoded in its small-scale structures. 

Random geometric graphs (RGGs) provide a natural mathematical framework in which to study such questions. By construction, their connectivity is determined entirely by the spatial proximity of nodes, making them a foundational model for studying spatially embedded networks. The commute time between nodes is an essential measure on networks, yet the relationship between commute times and the underlying geometry of spatially constrained networks remains poorly understood in the sparse regimes most relevant to real-world applications. The present study aims to motivate further work to help fill that gap, with the aim of clarifying when and how local and geometric structure governs global random-walk distances in RGGs.

\subsection{Commute Times and Applications}

Let $G = (V, E)$ be a finite, connected, undirected 
graph (below we will also consider weighted graphs) and let $(X_t)_{t \ge 0}$ denote the simple random walk on $G$. The \emph{hitting time} from $u$ to $v$ is
$H_{uv} \equiv \mathbb{E}\!\left[\,\min\{t \ge 0 : X_t = v\} \,\middle|\, X_0 = u\right],$
the expected number of steps for the walk to first reach $v$ starting from $u$. The \emph{commute time} between $u$ and $v$ is then $C_{uv} \equiv H_{uv} + H_{vu}$, the expected number of steps to travel from $u$ to $v$ and back. The commute time defines a metric on $V$. Unlike the geodesic distance, which measures the shortest path between two nodes, the commute time represents an average notion of distance that accounts for a broader array of paths and structures of the graph.

The most common physical interpretation of commute time is effective resistance: the voltage difference between two nodes under a unit current flow from one to the other, defined precisely below in \cref{def:eff-res}. The commute times $C_{uv}$ and effective resistances $E_{uv}$ are proportional \cite{chandra1996electrical}, with $C_{uv} = \mathrm{Vol}(G)\, E_{uv}$, where $\mathrm{Vol}(G)$ is the sum of the (weighted) degrees of nodes of $G$. Utilizing this relationship, we will interchangeably discuss commute times and effective resistances throughout.

The commute time metric has many applications in addition to its link to effective resistance. In network science, it has been used in community detection problems \cite{yen2007graph}, bottleneck detection algorithms \cite{spielman2004nearly}, and routing and network design theory \cite{ghosh2008minimizing}. In machine learning, the commute time has been used for dimensionality reduction \cite{fouss2004novel} and recommendation algorithms \cite{fouss2007random}. The commute time has also been used for molecular identification/description \cite{klein1993resistance} and diffusion modeling \cite{coifman2006diffusion}.

\subsection{Random Graph Models and Random Geometric Graphs}

Random graph models provide natural frameworks for studying large networks. One of the simplest such models, due to Erd\H{o}s and R\'enyi \cite{erdos1959random}, places an edge between each pair of $n$ vertices independently and identically distributed (i.i.d.)\ with probability $p$. Despite its simplicity, this model has been instrumental for studying phase transitions in the $n\to\infty$ thermodynamic limit, with the emergence of a giant component for $p>1/n$ \cite{bollobas2001random}. Other random graph models have been introduced to capture different structural features of real-world networks: the Barab\'asi--Albert model \cite{barabasi1999emergence} reproduces heavy-tailed degree distributions, the Watts--Strogatz model provides a minimal process to generate ``small world'' network properties, and various versions of the stochastic block model \cite{holland1983stochastic} can be used to encode and generate graphs with community structure. 

In contrast to the above models, random geometric graph models explicitly generate from and account for spatial locality in a physically embedded space. Introduced by Gilbert \cite{gilbert1961random}, random geometric graphs (RGGs) are determined by their local geometry. The Gilbert disk model is characterized by a metric space domain $(X,d)$, number of nodes $n$, and radius of connection $r$ \cite{penrose2003random}. Specifically, $n$ points are placed randomly according to a given distribution in $X$ and connected by edges wherever the distances, defined by the metric $d$, between two nodes are smaller than $r$.

Throughout this paper, we work with the {binomial} RGG with points uniformly i.i.d.\ on the two-dimensional unit torus $\mathbb{T}^2 \equiv \mathbb{R}^2 / \mathbb{Z}^2$ endowed with the standard (shortest distance) Euclidean metric. In so doing, we conveniently avoid the boundary effects present in RGGs on the non-periodic square. We write $G(X_n, r_n)$ for the graph of $n$ nodes at points $X_n = \{\xi_1, \dots, \xi_n\}$ drawn independently and uniformly from $\mathbb{T}^2$, with edges placed between pairs $(u, v)$ that satisfy $d(u,v) < r_n$. The expected vertex degree is then $(n-1)\pi r_n^2 \approx n\pi r_n^2$ for large $n$. Many asymptotic results in the literature are first established for the Poissonized counterpart $G(\mathcal{P}_{\mu_n}, r_n)$, in which the vertex set is a Poisson point process of intensity $\mu_n \sim n$, and then transferred to the binomial model by standard coupling arguments (see \cite{penrose2003random}).

Various extensions and generalizations of the Gilbert disk model exist, such as soft RGGs \cite{penrose2016connectivity}, also sometimes referred to as random distance graphs \cite{zhukovskii2010rdg}, that replace the hard radius of connectivity with probability distributions that decay with distance. The current study focuses solely on the uniform Gilbert disk model as defined above; all ``RGG'' references herein imply this model unless explicitly stated otherwise.

\subsection{Key Properties of RGGs}

When studying RGGs, it is natural to ask how connectedness depends on the number of nodes $n$ and radius of connection $r$ \cite{penrose2003random}, especially in the thermodynamic limit, $n \to \infty$, with some appropriate rescaling of the connection radius $r_n = r(n)$. Without loss of generality, we use a standard approach of a specified spatial domain $X$ of unit area, so that $n$ is both the number and density of nodes. 
As noted above, we focus our attention throughout the present work on uniformly distributed points on the two-dimensional unit torus, $X = \mathbb{T}^2$.

The first major result of relevance to our study concerns the emergence of a giant connected component (GCC). Penrose's analysis, carried out first in the Poissonized setting and then transferred to the binomial model by standard coupling arguments, yields the following threshold theorem.

\begin{theorem}[Giant component; Penrose \cite{penrose2003random,penrose2022largec}]
\label{thm:RGG-Giant}
For uniformly distributed points in a $D$-dimensional space, if $n r_n^D \to \lambda \in (0,\infty)$, then the proportion of vertices in the largest connected component (LCC) of the RGG converges in probability to $p_\infty(\lambda)$, with a critical value $\lambda_c$ s.t.\ $p_\infty(\lambda) = 0$ for $\lambda \le \lambda_c$ and $p_\infty(\lambda) > 0$ for $\lambda > \lambda_c$.
\end{theorem}

Theorem \labelcref{thm:RGG-Giant} describes the thermodynamic limit to fixed mean degree, $k$, with $k=\pi\lambda$ in $D=2$. A numerical estimate of the critical area fraction $\phi_c\doteq 0.676339$ for penetrable discs in the plane by Quintanilla et al.\ \cite{Quintanilla2000} is equivalent, via $\phi_c = 1-e^{-\pi\lambda_c/4}$ \cite{penrose2003random}, to a critical mean degree of $k_c \doteq 4.51223$. Below this value, the LCC includes a vanishingly small fraction of the nodes; above this value, there is a GCC including $0<p_\infty(\lambda)<1$ fraction of the nodes. 

A second major result characterizes when the graph becomes fully connected. Again, Penrose establishes this first for the Poissonized model and then de-Poissonizes.

\begin{theorem}[Connectivity threshold; Penrose \cite{penrose2003random}]
\label{thm:RGG-CON}
For uniformly distributed points in a $D$-dimensional space, let $r_n > 0$ and suppose $n r_n^D / \ln n \to \alpha \in (0, \infty)$. There exists an $\alpha_c$ s.t.\ if $\alpha > \alpha_c$, then an arbitrary node of $G(X_n, r_n)$ is almost surely in the GCC. If $\alpha < \alpha_c$ the RGG is disconnected with high probability.
\end{theorem}

Summarizing for $D=2$, wherein $\alpha_c=1/\pi$, the above theorems identify three connectivity regimes in the thermodynamic limit:
\smallskip
\begin{enumerate}
    \item RGGs are \emph{subcritical} if $(n-1)\pi r_n^2 < k_c$. The mean degree is below the percolation threshold so no GCC exists in the thermodynamic limit and the graph consists of small disconnected components.
    \item RGGs are \emph{supercritical} if $(n-1)\pi r_n^2 > k_c$. The mean degree exceeds the percolation threshold and the GCC includes a positive fraction of vertices.
    \item RGGs are \emph{fully connected} if $(n-1)\pi r_n^2 > \ln n$. The graph is connected with high probability. If $(n-1)\pi r_n^2 \sim \ln n$ and $(n-1)\pi r_n^2/\ln n >1$, we will refer to this subregime as \emph{sharply fully connected}.
\end{enumerate}

\subsection{Commute Times and RGGs}

Having established the three connectivity regimes for RGGs, we briefly review the known behavior of commute times in each, starting with the \emph{fully connected} regime where the most relevant result for our current study is due to von Luxburg et al.~\cite{vonLuxburg2014}: \emph{``Loosely speaking, the convergence results say that whenever the graph is reasonably large, the degrees are not too small, and the bottleneck is not too extreme, then the commute distance between two vertices can be approximated by the sum of their inverse degrees.''} (See also \cite{radl2009}.) That is, in very dense RGGs, the commute time $C_{uv}$ degenerates in the sense of losing sensitivity to global structures of the graph and embedding space, such as clusters, bottlenecks, and details of long-range paths. We restate the result most relevant to our current study, as restricted to our focus on uniformly distributed points on $\mathbb{T}^2$, stressing that the theorems in \cite{vonLuxburg2014} are much more general than considered here.

\begin{theorem}[Theorem 3 of von Luxburg et al.\ \cite{vonLuxburg2014} as applied to uniformly distributed points on $\mathbb{T}^2$]\label{thm:vonLuxburg} Let $G_n = (V_n, E_n)$ be an RGG with radius of connection $r_n$ from $n$ uniformly distributed points on $\mathbb{T}^2$ with vertex degrees $D_u$ and graph volume $\mathrm{Vol}(G_n) = \sum_{u\in V_n} D_u$. Suppose that the fixed pair of vertices $u$ and $v$ at positions $X_u$ and $X_v$ are separated by Euclidean distance $d(X_u,X_v) > 8r_n$. Assuming that $n r_n^2 / \ln n \to \infty$ as $n\to\infty$, there exists a constant $c$ such that the commute distance satisfies, with probability converging to $1$, 
\[
nr_n^2\left\vert\frac{1}{\mathrm{Vol}(G_n)}C_{uv} -
\left(\frac{1}{D_u}+\frac{1}{D_v}\right)\right| \leq c/nr_n^3\,.
\]
Therefore, further requiring that $r_n\to 0$ and $nr_n^3\to\infty$ yields
\[
C_{uv} \;\to\; \mathrm{Vol}(G_n)\left(\frac{1}{D_u} + \frac{1}{D_v}\right)\qquad almost\;surely\,.
\]
\end{theorem}

That is, under these dense-graph conditions, the limiting behavior of the commute time $C_{uv}$ depends only on the sum of the inverses of degrees $D_u$ and $D_v$. However, the conditions of von Luxburg et al.'s theorem are not satisfied in the sharply fully connected setting $(n-1)\pi r_n^2 \sim \ln n$, let alone on the GCC in the supercritical (but not fully connected) regime. Indeed, the nature of the relationship between commute time and Euclidean distance in these other regimes appears to be an open question.

At the other extreme, the \emph{subcritical} regime, the graph is fragmented into small disconnected components, and the commute time between vertices in distinct components is infinite while the commute time between two nodes in the same connected component is necessarily subject to the structural details of that component.

Between these extremes, including the \emph{supercritical} (but not fully connected) and the \emph{sharply fully connected} regimes, we presume that the limiting behavior of the commute times should be impacted by both local degree information and spatial properties of the graph in its embedding region. But to the best of our knowledge, no such characterization of commute times for RGGs in either of these regimes exists. The central motivating goal of the present work is to explore the asymptotic behavior of commute times on the largest connected components of RGGs in the \emph{supercritical} regime, to understand the relationship between random-walk-based distances and Euclidean geometry in large, spatially embedded networks.

The rest of this paper is organized as follows. In \autoref{sec:res-com}, we review the essential definitions and computations for hitting times and commute times (and, thereby, effective resistances) in terms of the graph Laplacian. In \autoref{sec:model}, we identify different terms we expect to appear in formulae for effective resistance in spatially-homogenized and square-lattice geometries. As described in \autoref{sec:numerics}, we then generate large ensembles of RGGs across a range of graph sizes $n$ and expected mean degrees $k$, compute commute times on the LCCs, and fit regression models to estimate the dependence on geometrical and local graph features. To further understand how the geometry of the torus influences our results in a more controlled setting, in \autoref{sec:perturb} we examine how commute times change under perturbations of square lattices embedded in $\mathbb{T}^2$. We conclude in \autoref{sec:conclusions} with thoughts for possible future directions.

\section{Commute Times and Effective Resistances}
\label{sec:res-com}

For concreteness, we briefly review how to compute hitting times and commute times on connected weighted undirected graphs in terms of the graph Laplacian matrix and the formal relationship between commute times and effective resistances.

\subsection{Computation of Hitting Times and Commute Times}
\label{subsec:computation}
Let $G = (V, E, w)$ be an undirected, connected, simple weighted graph with vertices (also called `nodes') $V$, edges $E$, and positive weights $w \in \mathbb{R}_{>0}^{|E|}$. Let $\mathbf{A}$ be the weighted adjacency matrix of $G$, with entry $A_{uv}$ equal to the weight $w_{(u,v)}$ of the edge $(u,v)$ between $u$ and $v$ (if present) and $0$ otherwise. Define $\mathbf{D}$ as the diagonal matrix of the weighted degrees of $G$ with ${D}_{uu}=\deg(u)\equiv\sum_v A_{uv}$ and $D_{uv}=0$ for $u\neq v$, emphasizing that we use $\deg(u)$ to be the \emph{weighted} degree of $u$. For notational convenience, let $\vec{\mathbf{1}}_u$ be the (column) vector of $0$s except a single $1$ in the $u$th entry and $\vec{\mathbf{D}}$ be the vector of weighted degrees (i.e., $D_u=D_{uu}$). Unweighted graphs are treated by setting weights to $1$. The following results and arguments are reproduced from Sidford \cite{sidford2018lecture13} with only minor changes in notation. 

\begin{definition}[Random Walk]
A random sequence $X_1, ..., X_k$ of vertices in $G$ is a random walk if independently for all $i \in [k-1]$:
\begin{equation*}
    \Pr\left(X_{i+1} = v \mid X_i = u\right) = 
    \begin{cases} 
        \frac{w_{(u,v)}}{\deg(u)} & \text{if } (u,v) \in E\,, \\
        0 & \text{otherwise}.
    \end{cases}
\end{equation*}
Furthermore, let $\mathbf{W} = \mathbf{A}\mathbf{D}^{-1}$ so that 
$\mathbf{W}_{vu} = \Pr\left(X_{i+1} = v \mid X_i = u\right)$. Then
$\mathbf{W}\vec{\mathbf{1}}_u$ is the distribution of vertex locations after one step of a random walk starting from $u$.
\end{definition}

\begin{definition}[Hitting Time]\label{def:hitting}
Define the hitting time $H_{uv}$ between nodes $u$ and $v$ to be the expected number of steps for a random walk started at $u$ to first reach $v$, i.e., $H_{uv} \equiv \mathbb{E}\!\left[\,\min\{t \ge 0 : X_t = v\} \,\middle|\, X_0 = u\right]$. Note that $H_{uu} = 0$.
\end{definition}

With a goal of defining a distance between nodes, and recognizing that $H_{uv}$ and $H_{vu}$ are not in general equal to one another, one can simply obtain a symmetric measure between a pair of nodes by summing the two hitting times.

\begin{definition}[Commute Time]\label{def:commute}
Define the commute time between nodes $u$ and $v$ as $C_{uv} = H_{uv}+H_{vu}$, i.e., the expected number of steps for a random walk started at $u$ to first reach $v$ and then return to $u$. Unlike hitting times, commute times are symmetric by construction: $C_{uv} = C_{vu}$.
\end{definition}

\begin{definition}[Graph Laplacian and its Pseudoinverse] Define the (combinatorial) graph Laplacian matrix of $G$ as $\mathbf{L} \equiv \mathbf{D}-\mathbf{A}$. Additionally, we use the notation $\mathbf{L}^+$ for the Moore--Penrose pseudoinverse of the graph Laplacian. 
\end{definition}

\begin{theorem}[Hitting Time Calculation]\label{thm:hit_tim} 
For any $u,v\in V$, the hitting time $H_{uv}$ from $u$ to $v$ satisfies
\begin{equation*}H_{uv} = (\vec{\mathbf{1}}_u - \vec{\mathbf{1}}_v)^\top \mathbf{L}^+ (\vec{\mathbf{D}} - 
\mathrm{Vol}(G)\, \vec{\mathbf{1}}_v).
\end{equation*}
\end{theorem}
\begin{proof}
This proof relies on the recursive fact that for any distinct vertices $u,v \in V$, the hitting time from $u$ to $v$ is $1$ plus the weighted expected hitting time from the neighboring vertices of $u$ to $v$. This can be written as:
\begin{equation}
\label{eq:hitting-time-recursive}
H_{uv} = 1 + \sum_{a\in \mathcal{N}(u)} \mathbf{W}_{au}H_{av}\,.
\end{equation}
For ease of argument, the following notation is introduced, fixing $v\in V$ and defining $\vec{\mathbf{h}}\in \mathbb{R}^V$ to be such that, for all $a \in V$, ${h}_a = H_{av}$. Notice ${h}_v = 0$. Using the new notation, \cref{eq:hitting-time-recursive} can be rewritten for $u \neq v$ as 
\begin{equation}
\label{eq:hitting-time-recursive-new-notation}
\vec{\mathbf{h}}^\top \vec{\mathbf{1}}_u = 1 + \vec{\mathbf{h}}^\top \mathbf{W} \vec{\mathbf{1}}_u\,.
\end{equation}
Define $c\equiv - \vec{\mathbf{h}}^\top \mathbf{W} \vec{\mathbf{1}}_v$ and then notice that
\begin{equation}
\label{eq:hitting-time-recursive-new-notation-t}
\vec{\mathbf{h}}^\top \vec{\mathbf{1}}_v = c + \vec{\mathbf{h}}^\top \mathbf{W} \vec{\mathbf{1}}_v\,.
\end{equation}
This constraint will be used later to define a system to solve for $\vec{\mathbf{h}}$. Combining equations \cref{eq:hitting-time-recursive-new-notation} and \cref{eq:hitting-time-recursive-new-notation-t} into matrix form we have
\begin{equation}
\label{eq:hitting-time-recursive-matrix-form}
(\mathbf{I} - \mathbf{W}^\top)\vec{\mathbf{h}} = \vec{\mathbf{1}} - (1-c)\vec{\mathbf{1}}_v
\end{equation}
where $\vec{\mathbf{1}}$ is the vector of all $1$s.
Recalling that $\mathbf{A}=\mathbf{A}^\top$ for an undirected graph, $\mathbf{D}(\mathbf{I} - \mathbf{W}^\top) = \mathbf{D}(\mathbf{I} - \mathbf{D}^{-1}\mathbf{A}) = \mathbf{L}$, from which \cref{eq:hitting-time-recursive-matrix-form} gives
\begin{equation}
\label{eq:hitting-time-recursive-matrix-form-2}
\mathbf{L}\vec{\mathbf{h}} = \vec{\mathbf{D}} - {D}_v(1-c)\vec{\mathbf{1}}_v\,.
\end{equation}
In order to then compute $\vec{\mathbf{h}}$, recall that since $G$ is connected the kernel of $\mathbf{L}$ is $\ker(\mathbf{L}) = \text{span}(\vec{\mathbf{1}})$ \cite{chung1997spectral} and, since $\mathbf{L}$ is symmetric, $\text{im}(\mathbf{L})$ is the orthogonal complement of $\ker(\mathbf{L})=\text{span}(\vec{\mathbf{1}})$, so \cref{eq:hitting-time-recursive-matrix-form-2} gives
\begin{equation}
\mathbf{L}\vec{\mathbf{h}} = \vec{\mathbf{D}} - {D}_v(1-c)\vec{\mathbf{1}}_v \;\Rightarrow\; 0 = \vec{\mathbf{1}}^\top(\vec{\mathbf{D}} - {D}_v(1-c)\vec{\mathbf{1}}_v) \;\Rightarrow\; {D}_v(1-c) = \|\vec{\mathbf{D}}\|_1\,,
\end{equation}
where we note our notation includes the following trivial identities for the sum of the weighted degrees:
\begin{equation}
\mathrm{Vol}(G) = \sum_u D_u = \|\vec{\mathbf{D}}\|_1 = \mathrm{Tr}(\mathbf{D})\,.
\end{equation}
Then, again using $\ker(\mathbf{L})=\text{span}(\vec{\mathbf{1}})$, for any $\vec{\mathbf{x}}\in \mathbb{R}^V$ satisfying 
\begin{equation}
\label{eq:linearsystem_x}
\mathbf{L}\vec{\mathbf{x}} = \vec{\mathbf{D}}-\mathrm{Vol}(G)\vec{\mathbf{1}}_v    
\end{equation} 
there exists some $c\in \mathbb{R}$ such that $\vec{\mathbf{h}} = \vec{\mathbf{x}} + c\vec{\mathbf{1}}$. Recalling that ${h}_v = 0$ then requires $c=-\mathbf{x}_v$. 
Finally, setting $\vec{\mathbf{x}}= \mathbf{L}^+(\vec{\mathbf{D}} - \mathrm{Vol}(G)\vec{\mathbf{1}}_v)$ gives
\begin{align}
\label{eq:hitting_x}
H_{uv} = \vec{\mathbf{1}}_u^\top \vec{\mathbf{h}} = \vec{\mathbf{1}}_u^\top(\vec{\mathbf{x}}-x_v \vec{\mathbf{1}}) &= x_u - x_v \\ &= (\vec{\mathbf{1}}_u - \vec{\mathbf{1}}_v)^\top \mathbf{L}^+(\vec{\mathbf{D}} - \mathrm{Vol}(G)\vec{\mathbf{1}}_v)\,.
\end{align}
\end{proof}

\begin{theorem}[Commute Time Calculation]\label{thm:com_tim} 
For any $u,v\in V$, the commute time $C_{uv}$ between $u$ and $v$ satisfies
\begin{equation*}
C_{uv} = \mathrm{Vol}(G) (\vec{\mathbf{1}}_u-\vec{\mathbf{1}}_v)^\top \mathbf{L}^+(\vec{\mathbf{1}}_u-\vec{\mathbf{1}}_v) = \mathrm{Vol}(G) \left({L}^+_{uu} + {L}^+_{vv} - 2{L}^+_{uv}\right) \,.
\end{equation*}
\end{theorem}
\begin{proof}
The result follows directly from \cref{thm:hit_tim}.
\end{proof}

We note that the calculation of $\mathbf{L}^+$ is streamlined by the properties of $\mathbf{L}$ as a real symmetric matrix (specifically, that its eigenvectors are orthogonal) and once again relying on $\ker(\mathbf{L})=\text{span}(\vec{\mathbf{1}})$ for $G$ a single connected component. Then one can simply shift $\mathbf{L}$ by a multiple of the matrix of all $1$s, $\mathbf{J}=\vec{\mathbf{1}}\vec{\mathbf{1}}^\top$, to shift the $0$ eigenvalue without changing any of the other eigenvalue/vector pairs (by orthogonality), take a (regular) inverse, and then un-do the shift \cite{bonchi2012}. Noting that $\frac{1}{n}\mathbf{J}\vec{\mathbf{1}} = \vec{\mathbf{1}}$ then yields
\begin{equation}
    \mathbf{L}^+ = \left( \mathbf{L}+\frac{1}{n}\mathbf{J} \right)^{-1} -\; \frac{1}{n}\mathbf{J}\,,
\end{equation}
so that $\mathbf{L}^+\vec{\mathbf{1}}=\vec{\mathbf{0}}$ as required. While this full inverse calculation gives all pairwise commute times, it nevertheless requires $O(n^3)$ operations to calculate and $O(n^2)$ memory to hold the full $\mathbf{L}^+$. An alternative calculation for a single pair is to perform the sparse $O(n)$ linear solve in \cref{eq:linearsystem_x}, use \cref{eq:hitting_x} for $H_{uv}$, and repeat swapping $u$ and $v$; but this $O(n)$ calculation for each unique sink node must be recalculated to obtain the commute time between other nodes. (See \cite{bonchi2012} for further discussion.)

\subsection{Relating Effective Resistance to Commute Time} The following results are compiled from \cite{chandra1996electrical}, \cite{doyle1984random}, and \cite{spielman2019spectral}. 

\begin{definition}[Effective Resistance]
\label{def:eff-res}
Let $u, v \in V$. Interpret $G$ as an electrical network where each edge 
$(i,j) \in E$ is a resistor with resistance $w_{ij}^{-1}$. Inject one unit of current at $u$ and extract one unit at $v$, so the current demand vector is:
\begin{equation*}
\vec{\mathbf{b}} = \vec{\mathbf{1}}_u - \vec{\mathbf{1}}_v.
\end{equation*}
To see how the graph Laplacian $\mathbf{L}$ governs this network, note that when assigning voltages $\vec{\boldsymbol{\phi}} \in \mathbb{R}^V$ to the vertices, Ohm's law gives the current flowing from vertex $i$ to vertex $j$ as $f_{ij} = w_{ij}(\vec{\boldsymbol{\phi}}_i - \vec{\boldsymbol{\phi}}_j)$. Kirchhoff's current law then requires that the net current flowing out of each vertex $i$ equals the external demand $\mathbf{b}_i$:
\begin{equation*}
\sum_{j \in \mathcal{N}(i)} w_{ij}(\vec{\boldsymbol{\phi}}_i - \vec{\boldsymbol{\phi}}_j) = \mathbf{b}_i \quad \forall\, 
i \in V.
\end{equation*}
The left-hand side is precisely $[\mathbf{L}\vec{\boldsymbol{\phi}}]_i$ since 
$\mathbf{L} = \mathbf{D} - \mathbf{A}$, so Ohm's law and Kirchhoff's current 
law together reduce to the Laplacian system
\begin{equation*}
\mathbf{L}\vec{\boldsymbol{\phi}} = \vec{\mathbf{b}}  = \vec{\mathbf{1}}_u - \vec{\mathbf{1}}_v.
\end{equation*}
Since $\vec{\mathbf{b}} \perp \vec{\mathbf{1}}$ and $\ker(\mathbf{L}) = 
\text{span}(\vec{\mathbf{1}})$, similarly to \cref{thm:hit_tim} this system has a solution given by 
$\vec{\boldsymbol{\phi}} = \mathbf{L}^+(\vec{\mathbf{1}}_u - \vec{\mathbf{1}}_v)$. 
The effective resistance is then defined as the resulting voltage difference 
between $u$ and $v$:
\begin{equation*}
E_{uv} \equiv \vec{\boldsymbol{\phi}}_u - \vec{\boldsymbol{\phi}}_v = (\vec{\mathbf{1}}_u - \vec{\mathbf{1}}_v)^\top 
\vec{\boldsymbol{\phi}} = (\vec{\mathbf{1}}_u - \vec{\mathbf{1}}_v)^\top \mathbf{L}^+ 
(\vec{\mathbf{1}}_u - \vec{\mathbf{1}}_v).
\end{equation*}
\end{definition}

\begin{theorem}[Commute Time and Effective Resistance]\label{thm:com_res}
For any $u, v \in V$, the commute time and effective resistance are related by
\begin{equation*}
C_{uv} = \mathrm{Vol}(G)\, E_{uv}\,.
\end{equation*}
\end{theorem}
\begin{proof}
The result follows immediately from \cref{thm:com_tim} and \cref{def:eff-res}.
\end{proof}

\subsection{Commute Time is a Metric}
Having reviewed the above definitions and calculations, it is important to our work below that the commute time is a metric between the vertices of the graph. That is, $C_{uv}\geq 0$, $C_{uv}=0 \iff u=v$, $C_{uv}=C_{vu}$, and $C_{uv}\leq C_{uw}+C_{wv}$, as all follow from \cref{def:hitting} and \cref{def:commute}. The first three properties also follow immediately from \cref{thm:com_tim} and the positive semidefiniteness of $\mathbf{L}$. For proofs of the triangle inequality from the properties of $\mathbf{L}$, we refer the interested reader to  \cite{klein1993resistance,Tetali_1991}.

In contrast, it is of absolutely no importance whatsoever to our work below, but nevertheless amusing  \cite{Coppersmith_Tetali_Winkler_1993} that for any $v_1,v_2,\ldots,v_p\in V$, the multipoint commute sum of the hitting times through these nodes in order then returning to $v_1$ is the same as in the reverse order. That is, even though $H_{uv}$ is not typically equal to $H_{vu}$,
\begin{equation}
    H_{v_1v_2}+H_{v_2v_3}+\cdots+H_{v_{p-1}v_p}+H_{v_pv_1} = 
    H_{v_1v_p}+H_{v_{p}v_{p-1}}+\cdots+H_{v_3v_2}+H_{v_2v_1}\,,
\end{equation}
as follows directly from \cref{thm:hit_tim}. (The trivial $p=2$ case recovers $C_{uv}=C_{vu}$.)

\section{Building a Model for Effective Resistance}
\label{sec:model}

Our goal for the remainder of this paper is to understand the interplay between the commute times between vertices of an RGG and the physical embedding of their points on the torus. A natural starting point for building such a model is to first understand the behavior of effective resistance on one of the simplest graphs in the same embedding space: a regular square lattice with axes aligned with those of $\mathbb{T}^2$. The square lattice provides a tractable setting to identify, from first principles, spatial features that should appear in a model for commute times, before we try to extend that model to the irregular geometry of an RGG. 
We will then combine these features with the degree dependence for dense graphs in von Luxburg et al.'s theorem~\cite{vonLuxburg2014} (\cref{thm:vonLuxburg}).

\subsection{Effective Resistance on a Square Lattice: Analytical Results}
\label{sec:lattice-analytical}

Effective resistances between nodes on infinite square lattices of unit resistors are well documented in the literature. Cserti \cite{cserti2000} derives exact expressions using lattice Green's functions and Jeng \cite{jeng2004} extends these results to finite toroidal and cylindrical lattices with a superposition argument. We  directly cite the necessary results by Rossi et al.~\cite{rossi2015} for effective resistances in $\mathbb{T}^2$. We build on these results by proposing the following terms for a truncated asymptotic approximation for the effective resistance $E(d, \theta)$ between two nodes separated by Euclidean distance $d>0$ at angle $\theta$ formed between an axis of the lattice (and the torus) and the line through both nodes:
\begin{equation}
\label{eq:resistance-expansion}
    E(d, \theta) \approx \beta_0 + \beta_1\ln(d) 
    + \beta_2 d^2 + \beta_3 d^4\cos(4\theta)\,,
\end{equation}
where each of the $\beta_i$ coefficients might depend on $n$. Moreover, in the $n\gg 1$ limit of large square lattices we expect $\beta_1\to 1/\pi$ and $\beta_2\to -1/2$.
Reasoning associated with each of these 4 terms is provided below.

\subsubsection{The logarithmic term} The logarithmic dependence $(1/\pi)\ln(d)$ is established rigorously by Rossi et al.~\cite{rossi2015}. Specifically, using the connection between effective resistance and the eigenvalues of the graph Laplacian, they prove the following result for the $N \times N$ toroidal square grid.

\begin{theorem}[Rossi, Frasca \& Fagnani \cite{rossi2015}] Let $T_{N^2}$ be the two-dimensional toroidal grid with side length $N$ 
and average effective resistance $E_{\mathrm{avg}}(T_{N^2})$. Then
\[
    E_{\mathrm{avg}}(T_{N^2}) \;\sim\; \frac{1}{2\pi}\ln N 
    \quad \text{as } N \to \infty.
\]
\end{theorem}

This result relates to the fact that the effective resistance between two nodes separated by Euclidean distance $d$ on the two-dimensional toroidal grid grows $\sim\ln d$ at the level of individual pairs. The limiting behavior of the coefficient, $\beta_1\to\frac{1}{\pi}$ (rather than $\frac{1}{2\pi}$), arises because a factor of $\frac{1}{2}$ is absorbed in averaging resistance over all pairs; for a given source-sink pair at separation $d$, the source contributes $\frac{1}{2\pi}\ln d$ and the sink contributes an equal amount, giving $\frac{1}{\pi}\ln d$. The logarithmic scaling is a signature of two-dimensional diffusion, reflecting the fact that the Green's function of the Laplacian on $\mathbb{R}^2$ is logarithmic. 

At the same time, this logarithmic dependence leads to our $\beta_0$ term, since upon rescaling from the $N\times N$ toroidal square grid to the unit torus the logarithmic term becomes $\ln(d/r_n) = \ln(d) - \ln(r_n)$, the latter term being a constant with respect to $d$ and $\theta$ (but expected to depend on $n$, as we will see in \cref{sec:lattice-numerical}).

\subsubsection{The quadratic background term}
The $-\frac{1}{2}d^2$ term arises from the requirement of global charge neutrality in the absence of a boundary. Specifically, compensating with a uniformly distributed charge, the Green's function on the unit torus is the solution to
$\nabla^2 V = \delta(\mathbf{r}) - 1$.
At distances small enough to ignore the periodic boundary conditions (that is, also implicitly assuming $n\gg 1$), $\nabla^2 V = -1$ then yields the particular solution $V = -\frac{1}{4}d^2$, and accounting for both source and sink gives a net contribution of $-\frac{1}{2}d^2$ to the resistance. This is a purely finite-domain correction, absent on the infinite plane.

\subsubsection{The angular correction term} 
The logarithmic and quadratic terms both have isotropic symmetry: they depend only on the separation distance $d$. But the representation of  $\mathbb{T}^2$ as the unit square $[0,1]^2$ with periodic boundary conditions highlights the absence of this symmetry. Two nodes separated by the same distance $d$ but at different angles $\theta$ relative to the lattice axes sit in geometrically different environments: for example, traveling along an angle $\theta = 0$ (along an axis) reaches the periodic boundary more quickly than one at $\theta = \pi/4$ (along a diagonal). This anisotropy must be captured by additional angular correction.

After accounting for the source term and the topological (i.e., global charge neutrality) background, any remaining correction $V_{geom}$ must satisfy the homogeneous Laplace equation $\nabla^2 V_{geom} = 0$. The space of solutions to the homogeneous Laplace equation on a planar domain admits a natural orthogonal basis of harmonic polynomials. In polar coordinates these are $\{d^k\cos(k\theta),\, d^k\sin(k\theta)\}_{k \geq 0}$ and $\{d^{k}\cos(k\theta),\, d^{k}\sin(k\theta)\}_{k \leq -1}$, and any smooth correction away from the origin but far from the periodic boundaries can be written as a linear combination of these basis functions. We retain only the positive powers $k\geq 0$ because the negative powers diverge as $d \to 0$ whereas the effective resistance between nearby nodes should be dominated by the logarithmic term.

Any $k=0$ contribution can be absorbed into the $\beta_0$ constant. We further restrict the remaining harmonics to those that are consistent with the symmetry of the square torus.
First, the reflection symmetry requires that $E(d,\theta)$ is even in $\theta$, eliminating all $\sin(k\theta)$ terms. Second, the $90^\circ$ rotational symmetry of the square requires $E(d, \theta + \pi/2) = E(d,\theta)$, which for $\cos(k\theta)$ demands $k$ be a multiple of $4$. The only harmonics remaining can thus be written (redefining $k$) as $\{d^{4k}\cos(4k\theta)\}_{k \geq 1}$, and the leading-in-$d$ correction is the $k=1$ term
\[
    V_{geom}(d,\theta) = \beta_3 d^4\cos(4\theta)
\]
for some constant $\beta_3$ determined by the lattice geometry. This anisotropic term is largest in magnitude along the axes and diagonals of the square, where the regularity due to the anisotropy of the square geometry is most pronounced.
We ignore higher order terms $k \geq 2$ here as they vanish quickly for $d\ll 1$. 

Combining these terms yields the approximate expansion in \cref{eq:resistance-expansion}. Each nonconstant term encodes a distinct property of the graph: two-dimensional diffusion, the compact topology, and the square symmetry of the torus. This structural separation is precisely what makes this approximate expansion useful, in that the contribution of each term can be assessed independently.

\subsection{Effective Resistance on a Square Lattice: Numerical Results}
\label{sec:lattice-numerical}

To build and test our regression models for distances between vertices in RGGs, we wrote a lightweight Python library \cite{rupchin2025rgg} to facilitate the numerical study of RGGs and model fitting so that our numerical results are easily reproducible and verifiable.

Rather than computing effective resistances for all $\binom{n}{2}$ pairs of nodes, which as noted in \cref{subsec:computation} requires $O(n^3)$ operations and $O(n^2)$ memory, we estimate regression models fit to samples of pairs, with each pair requiring $O(n)$ operations. 

Because of the non-uniform distribution of pairwise distances on $\mathbb{T}^2$, uniform sampling of pairs is inappropriate. For two points drawn independently and uniformly from $\mathbb{T}^2$, the marginal density of their separation $d$ grows linearly at small $d$, peaks at an intermediate separation, and then decreases as $d$ approaches $d_{\max} = 1/\sqrt{2}$ due to the toroidal geometry. Uniform pair sampling would oversample intermediate distances and provide insufficient coverage of the short and long range cases, where the functional form of the resistance is more sensitive to the choice of features. To correct for this undersampling, we first discard all pairs with separation $d \leq r_n$, then partition the observed range of remaining separations (approximately $[r_n, d_{\max}]$) into 20 equally spaced bins and draw up to 50 pairs uniformly at random from each bin, yielding a sample of approximately 1,000 pairs with near-uniform marginal coverage over the sampled range. We chose the numbers of bins and samples to balance statistical certainty against computational cost.

\begin{table}[ht]
\centering
\begin{tabular}{lccccc}
\hline
$n$ & Intercept & $\ln(d)$ & $d^2$ & $d^4\cos(4\theta)$ & $R^2$ \\
\hline
$256$  & $1.3846$ & $0.3104$ & $-0.4672$ & $-0.2043$ & $0.99954$ \\
$1024$ & $1.6147$ & $0.3165$ & $-0.4891$ & $-0.2155$ & $0.99988$ \\
$4096$ & $1.8364$ & $0.3172$ & $-0.4919$ & $-0.2166$ & $0.99996$ \\
\hline
\vspace*{-0.15in}\\
\begin{tabular}{c}Theory\\ $n\to\infty$ \end{tabular} & \begin{tabular}{c} as $n\to 4n$, \\ increase by\\ $\frac{1}{\pi}\ln(2) \doteq 0.2206$ \end{tabular} & $\frac{1}{\pi} \doteq 0.3183$ & $-1/2$ & --- & --- \\
\vspace*{-0.12in}\\
\hline
\end{tabular}
\vspace*{0.05in}\\
\caption{Regression coefficients for effective resistance on $n$-node square lattice grids on $\mathbb{T}^2$, fitted using stratified distance-bin sampling. Theoretical values for the logarithmic and quadratic coefficients are derived in Section~\ref{sec:lattice-analytical}. The theoretical prediction for the change in the constant intercept as $n\to 4n$ follows from $r_n \to r_n/2$.}
\label{tab:square-lattice-no-perturbation}
\end{table}

We then fit \cref{eq:resistance-expansion} as a linear model to the distances between sampled pairs from square lattices on $\mathbb{T}^2$ at three graph sizes $n \in \{256, 1024, 4096\}$. The fitted coefficients and corresponding $R^2$ values are reported in Table~\ref{tab:square-lattice-no-perturbation}, with several features worth noting. The coefficient of the $\ln(d)$ term rapidly approaches $1/\pi$ as $n$ increases, recovering the theoretical prediction to within $0.4\%$ at $n = 4096$. At the same time, since $r_n \propto 1/\sqrt{n}$, the $\frac{1}{\pi}\ln(d/r_n)$ dependence should, for every increase in $n$ by a factor of $4$, increase the constant term by a value of $\frac{1}{\pi}\ln(2)$. Similarly, we observe that the coefficient of $d^2$ approaches $-1/2$, with the discrepancy shrinking from $6.6\%$ at $n = 256$ to $1.6\%$ at $n = 4096$. Meanwhile, the coefficient on the $d^4\cos(4\theta)$ term appears to stabilize near $-0.217$ for large $n$. Finally, we note that the $R^2$ values exceed $0.9995$ across all three graph sizes, confirming that the expansion presented in \cref{eq:resistance-expansion} captures almost all of the systematic variation in the effective resistance on the square lattice. We then take this model as a starting point for the RGG distance models developed below.

\section{Numerical Results for RGGs} 
\label{sec:numerics}
The following section includes descriptions of the methodology used to fit regression models for the effective resistances between nodes in RGGs, along with numerical results and additional comments.

\subsection{Model Specification}

The theoretical and numerical results of the preceding section for the square lattice suggest natural candidate terms to be included in a model for effective resistance on sparse RGGs. In the dense regime, von Luxburg et al.\ \cite{vonLuxburg2014} show that commute times degenerate to a function of local vertex degrees alone: for any fixed pair of vertices $u,v \in V_n$, as $n\to \infty$,
\[
C_{uv} \to \mathrm{Vol}(G_n)\,\Bigl(\frac{1}{D_u} + \frac{1}{D_v}\Bigr).
\]
While this approximation fails in the sparse supercritical regime and the sharply fully connected regime --- in particular, note that \cref{thm:vonLuxburg} requires even greater density than being fully connected --- the inverse degree sum $\frac{1}{D_u} + \frac{1}{D_v}$ still captures genuine local topological information about the graph and we suppose it may still contribute explanatory power even when it does not dominate. We therefore include it alongside the geometric terms from our square lattice analysis above to propose the following model for effective resistance between nodes $u$ and $v$ with (weighted) degrees $D_u$ and $D_v$, respectively, separated by Euclidean distance $d_{uv}$ in a sparse RGG:
\begin{equation}
\label{eq:rgg-model}
    E_{uv} \approx \beta_0 
    + \beta_1 \ln(d_{uv}) 
    + \beta_2 d_{uv}^2 
    + \beta_3 d_{uv}^4\cos(4\theta_{uv}) 
    + \beta_4 \left(\frac{1}{D_u} + \frac{1}{D_v}\right)\,,
\end{equation}
where, as previously defined above for the square lattice, $\theta_{uv}$ is the angle of the vector from $u$ to $v$ relative to an axis. We will estimate the coefficients $\beta_0, \dots, \beta_4$ by regression on a stratified sample of node pairs.

\subsection{Methodology}

For each combination of graph size $n$ and connection radius $r_n$, we generate an ensemble of $M$ (default $M=50$) independent RGG realizations $\{G^{(1)}, \dots, G^{(M)}\}$ on $\mathbb{T}^2$. For each realization we compute the effective resistances $E_{uv}$ for a stratified sample of node pairs in the LCC as described in 
Section~\ref{sec:lattice-numerical}, and assemble the feature vectors $\phi_{uv} = (\ln d_{uv},\, d_{uv}^2,\, d_{uv}^4\cos(4\theta_{uv}),\, \frac{1}{D_u} + \frac{1}{D_v})$ for each sampled pair $\{u,v\}$ separated by Euclidean distance $d_{uv}$ at angle $\theta_{uv}$ to the axes.

The ensemble of $M$ graphs with identical parameters but different random seeds is then partitioned into training and test sets: a random subset of $\lfloor s M \rfloor$ graphs (with $s = 0.5$ by default) contributes pairs to the training set, and the remaining graphs form the test set. Partitioning at the graph level rather than at the pair level ensures that the test set contains entirely unseen graph realizations, so that the reported test $R^2$ measures generalization across graphs rather than interpolation within a single graph. The regression model is fit to the training set and evaluated on the test set. To assess model stability, this train--test split is repeated over $S$ (default $S = 50$) independent random shuffles of the graph ensemble. Variation in the fitted coefficients and test $R^2$ across shuffles provides an empirical estimate of the sensitivity of the model to the particular graphs included in training versus testing. We report both the mean and standard deviation of each coefficient across shuffles.

This procedure is repeated across a range of graph sizes $n$ and connection radii $r_n$. We present findings for the supercritical regime with $(n-1)\pi r_n^2 > k_c$. Varying $n$ at fixed expected degree allows us to assess the finite-size behavior of the fitted coefficients and to examine whether they appear to converge to limiting values as $n \to \infty$, which would be consistent with the existence of a well-defined asymptotic approximation of the form \cref{eq:rgg-model}.

\subsection{Numerical Results for Supercritical RGGs}

\begin{figure}[ht]
    \centering
    \includegraphics[width=0.95\textwidth]{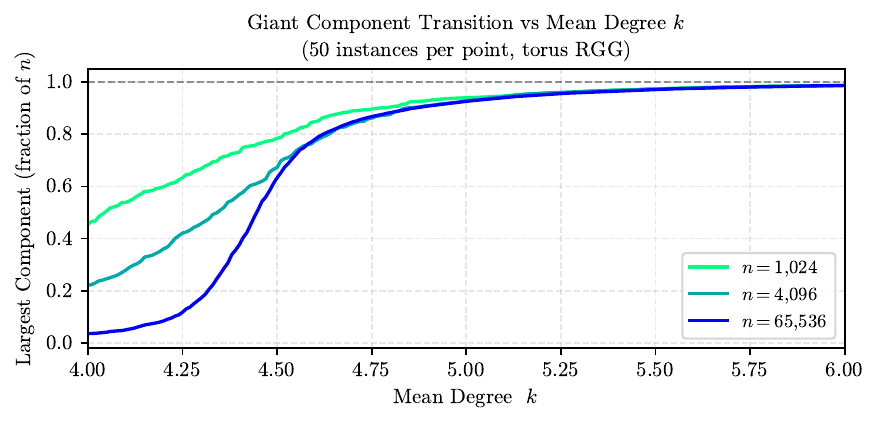}
    \caption{Largest connected component (LCC) size as a fraction of $n$, averaged over 50 independent RGG instances per $(n,k)$ pair, $n \in \{1024, 4096, 65536\}$. The transition sharpens with increasing $n$. The LCC contains $>90\%$ of all nodes for $k \ge 4.86$ across all system sizes.}
    \label{fig:lcc}
\end{figure}

\subsubsection{Giant Component Transition}

As a first validation of our implementation, we reproduce the giant component transition for toroidal RGGs. We generate the LCC size curves across mean degrees $k \in [4, 6]$ for system sizes $n \in \{1024, 4096, 65536\}$, averaging over $50$ independent instances per $(n, k)$ pair. As seen in Figure~\ref{fig:lcc}, the transition sharpens with increasing $n$, consistent with the theoretical picture of a phase transition as $n \to \infty$. Notably, for all system sizes tested, the LCC includes $>90\%$ of nodes by mean degree $k \approx 4.86$ (cf.\ the fully connected condition $k\sim\ln n$, with $\ln 2^{16}\doteq 11.1$). As such, for the graph sizes $n$ that we consider numerically, we do not expect the restriction to the LCC to be a confounding factor for any $k \geq 6$  in our regression results below. Instead, we hypothesize that most degradation in model fit we will observe at low-but-supercritical  $k$ is attributable to greater variability in graph geometry closer to $k_c$.

\subsubsection{Regression Results}

Table~\ref{tab:rgg-supercritical} reports the mean and standard 
deviation of the regression coefficients and test $R^2$ values across $50$ independent shuffles of the train--test split, for numbers of nodes $n$ between 512 and 8192 and expected mean degrees $k$ between 5 and 20. We visualize these results in Figures \ref{fig:Numerical-RGG-Supercritical-Fig-1}, \ref{fig:Numerical-RGG-Supercritical-Fig-2} and \ref{fig:Numerical-RGG-Supercritical-k=20-Coefs}.

\begin{table}[ht]
\centering
\small
\resizebox{\textwidth}{!}{
\begin{tabular}{llcccccc}
\hline
$n$ & $k$ & Test $R^2$ & $\ln(d)$ & $d^2$ & $d^4\cos(4\theta)$ & $(D_u^{-1}+D_v^{-1})$ & Intercept \\
\hline
$512$ & $5$ & $0.199 \pm 0.024$ & $2.205 \pm 0.115$ & $-2.794 \pm 0.306$ & $-1.005 \pm 0.390$ & $3.132 \pm 0.136$ & $5.526 \pm 0.323$ \\
$512$ & $6$ & $0.334 \pm 0.027$ & $0.819 \pm 0.027$ & $-1.327 \pm 0.080$ & $-0.563 \pm 0.109$ & $2.470 \pm 0.079$ & $1.915 \pm 0.070$ \\
$512$ & $8$ & $0.663 \pm 0.021$ & $0.258 \pm 0.004$ & $-0.385 \pm 0.023$ & $-0.133 \pm 0.029$ & $1.817 \pm 0.046$ & $0.545 \pm 0.015$ \\
$512$ & $10$ & $0.813 \pm 0.012$ & $0.136 \pm 0.002$ & $-0.213 \pm 0.009$ & $-0.088 \pm 0.009$ & $1.562 \pm 0.033$ & $0.270 \pm 0.008$ \\
$512$ & $12$ & $0.871 \pm 0.009$ & $0.083 \pm 0.001$ & $-0.127 \pm 0.004$ & $-0.054 \pm 0.005$ & $1.401 \pm 0.043$ & $0.160 \pm 0.007$ \\
$512$ & $16$ & $0.935 \pm 0.005$ & $0.041 \pm 0.000$ & $-0.060 \pm 0.002$ & $-0.025 \pm 0.002$ & $1.224 \pm 0.007$ & $0.079 \pm 0.001$ \\
$512$ & $20$ & $0.962 \pm 0.001$ & $0.024 \pm 0.000$ & $-0.037 \pm 0.001$ & $-0.015 \pm 0.001$ & $1.160 \pm 0.004$ & $0.046 \pm 0.001$ \\
\hline
$1024$ & $5$ & $0.207 \pm 0.024$ & $2.326 \pm 0.110$ & $-3.237 \pm 0.468$ & $-1.374 \pm 0.627$ & $3.254 \pm 0.104$ & $6.545 \pm 0.341$ \\
$1024$ & $6$ & $0.343 \pm 0.018$ & $0.820 \pm 0.019$ & $-1.363 \pm 0.088$ & $-0.615 \pm 0.108$ & $2.650 \pm 0.052$ & $2.153 \pm 0.051$ \\
$1024$ & $8$ & $0.673 \pm 0.012$ & $0.269 \pm 0.004$ & $-0.454 \pm 0.019$ & $-0.192 \pm 0.020$ & $1.826 \pm 0.026$ & $0.657 \pm 0.012$ \\
$1024$ & $10$ & $0.786 \pm 0.017$ & $0.136 \pm 0.002$ & $-0.226 \pm 0.007$ & $-0.108 \pm 0.007$ & $1.544 \pm 0.031$ & $0.321 \pm 0.008$ \\
$1024$ & $12$ & $0.853 \pm 0.026$ & $0.085 \pm 0.001$ & $-0.142 \pm 0.005$ & $-0.066 \pm 0.007$ & $1.423 \pm 0.051$ & $0.190 \pm 0.009$ \\
$1024$ & $16$ & $0.933 \pm 0.002$ & $0.042 \pm 0.000$ & $-0.068 \pm 0.002$ & $-0.033 \pm 0.002$ & $1.241 \pm 0.009$ & $0.094 \pm 0.001$ \\
$1024$ & $20$ & $0.958 \pm 0.001$ & $0.025 \pm 0.000$ & $-0.039 \pm 0.001$ & $-0.017 \pm 0.001$ & $1.178 \pm 0.005$ & $0.053 \pm 0.001$ \\
\hline
$2048$ & $5$ & $0.232 \pm 0.010$ & $2.653 \pm 0.069$ & $-3.715 \pm 0.331$ & $-1.409 \pm 0.504$ & $3.478 \pm 0.077$ & $8.144 \pm 0.214$ \\
$2048$ & $6$ & $0.359 \pm 0.038$ & $0.878 \pm 0.024$ & $-1.610 \pm 0.108$ & $-0.808 \pm 0.149$ & $2.595 \pm 0.042$ & $2.598 \pm 0.067$ \\
$2048$ & $8$ & $0.673 \pm 0.011$ & $0.261 \pm 0.004$ & $-0.396 \pm 0.021$ & $-0.135 \pm 0.024$ & $1.892 \pm 0.023$ & $0.717 \pm 0.010$ \\
$2048$ & $10$ & $0.806 \pm 0.010$ & $0.136 \pm 0.001$ & $-0.217 \pm 0.007$ & $-0.098 \pm 0.008$ & $1.552 \pm 0.018$ & $0.368 \pm 0.004$ \\
$2048$ & $12$ & $0.877 \pm 0.004$ & $0.084 \pm 0.001$ & $-0.139 \pm 0.004$ & $-0.067 \pm 0.004$ & $1.410 \pm 0.016$ & $0.221 \pm 0.003$ \\
$2048$ & $16$ & $0.938 \pm 0.002$ & $0.042 \pm 0.000$ & $-0.067 \pm 0.001$ & $-0.031 \pm 0.001$ & $1.243 \pm 0.005$ & $0.108 \pm 0.001$ \\
$2048$ & $20$ & $0.961 \pm 0.001$ & $0.025 \pm 0.000$ & $-0.039 \pm 0.001$ & $-0.017 \pm 0.001$ & $1.178 \pm 0.004$ & $0.062 \pm 0.000$ \\
\hline
$4096$ & $5$ & $0.198 \pm 0.021$ & $3.151 \pm 0.095$ & $-3.314 \pm 0.430$ & $-0.651 \pm 0.515$ & $3.627 \pm 0.105$ & $10.774 \pm 0.327$ \\
$4096$ & $6$ & $0.372 \pm 0.010$ & $0.869 \pm 0.012$ & $-1.481 \pm 0.066$ & $-0.753 \pm 0.102$ & $2.588 \pm 0.037$ & $2.901 \pm 0.036$ \\
$4096$ & $8$ & $0.647 \pm 0.015$ & $0.273 \pm 0.003$ & $-0.494 \pm 0.021$ & $-0.291 \pm 0.028$ & $1.919 \pm 0.023$ & $0.834 \pm 0.010$ \\
$4096$ & $10$ & $0.796 \pm 0.015$ & $0.137 \pm 0.001$ & $-0.232 \pm 0.009$ & $-0.112 \pm 0.012$ & $1.566 \pm 0.012$ & $0.416 \pm 0.004$ \\
$4096$ & $12$ & $0.853 \pm 0.030$ & $0.084 \pm 0.000$ & $-0.136 \pm 0.002$ & $-0.066 \pm 0.003$ & $1.408 \pm 0.018$ & $0.251 \pm 0.003$ \\
$4096$ & $16$ & $0.936 \pm 0.002$ & $0.041 \pm 0.000$ & $-0.066 \pm 0.001$ & $-0.031 \pm 0.002$ & $1.257 \pm 0.006$ & $0.120 \pm 0.001$ \\
$4096$ & $20$ & $0.958 \pm 0.001$ & $0.025 \pm 0.000$ & $-0.039 \pm 0.001$ & $-0.018 \pm 0.001$ & $1.188 \pm 0.004$ & $0.070 \pm 0.000$ \\
\hline
$8192$ & $5$ & $0.207 \pm 0.016$ & $3.180 \pm 0.087$ & $-3.591 \pm 0.485$ & $-0.072 \pm 0.589$ & $3.775 \pm 0.114$ & $11.883 \pm 0.298$ \\
$8192$ & $6$ & $0.358 \pm 0.011$ & $0.861 \pm 0.012$ & $-1.569 \pm 0.094$ & $-0.781 \pm 0.139$ & $2.659 \pm 0.043$ & $3.178 \pm 0.042$ \\
$8192$ & $8$ & $0.660 \pm 0.010$ & $0.263 \pm 0.003$ & $-0.430 \pm 0.022$ & $-0.193 \pm 0.030$ & $1.896 \pm 0.020$ & $0.911 \pm 0.009$ \\
$8192$ & $10$ & $0.810 \pm 0.010$ & $0.136 \pm 0.001$ & $-0.220 \pm 0.006$ & $-0.098 \pm 0.008$ & $1.557 \pm 0.012$ & $0.461 \pm 0.004$ \\
$8192$ & $12$ & $0.884 \pm 0.002$ & $0.084 \pm 0.000$ & $-0.136 \pm 0.003$ & $-0.062 \pm 0.005$ & $1.404 \pm 0.010$ & $0.280 \pm 0.002$ \\
$8192$ & $16$ & $0.939 \pm 0.001$ & $0.042 \pm 0.000$ & $-0.068 \pm 0.001$ & $-0.032 \pm 0.001$ & $1.256 \pm 0.005$ & $0.135 \pm 0.001$ \\
$8192$ & $20$ & $0.960 \pm 0.001$ & $0.025 \pm 0.000$ & $-0.041 \pm 0.001$ & $-0.018 \pm 0.001$ & $1.189 \pm 0.003$ & $0.079 \pm 0.000$ \\
\hline
\end{tabular}
}
\caption{Means and standard deviations of regression coefficients and test $R^2$ values across 50 independent train--test shuffles for each of the $(n,k)$ pairs tested here.}
\label{tab:rgg-supercritical}
\end{table}

\begin{figure}[ht]
    \centering
    \includegraphics[width=0.95\textwidth]{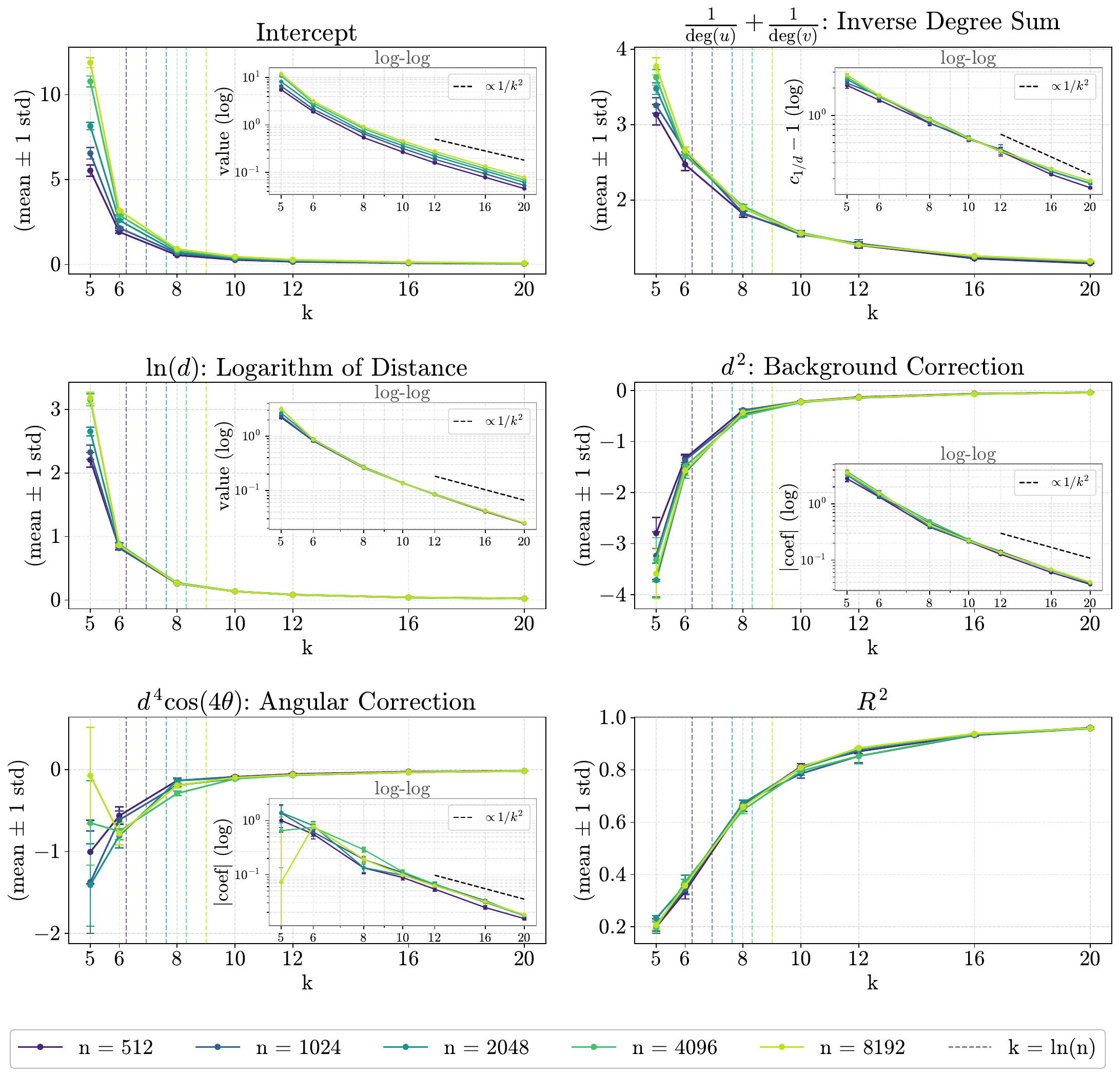}
    \caption{Visualization of Table~\ref{tab:rgg-supercritical} data plotting regression model fit and coefficients with respect to expected mean degree $k$ for various sizes $n$. Dashed vertical lines indicate the $k = \ln(n)$ fully-connected crossover. Insets: log-log plots with $1/k^2$ reference lines.}
    \label{fig:Numerical-RGG-Supercritical-Fig-1}
\end{figure}

\begin{figure}[ht]
    \centering
    \includegraphics[width=0.95\textwidth]{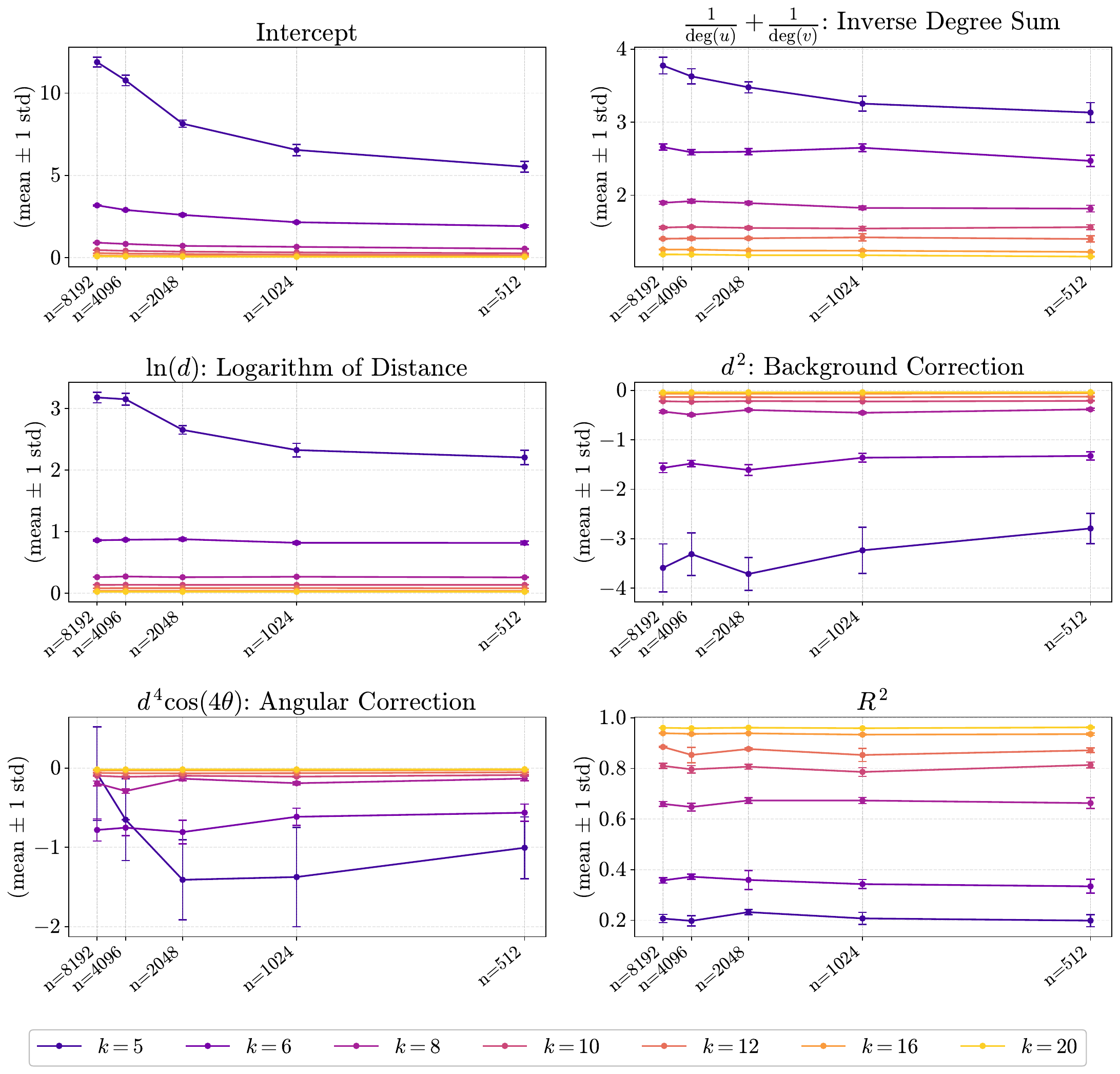}
    \caption{Visualization of Table~\ref{tab:rgg-supercritical} data plotting regression model fit and coefficients with respect to $1/n$ for various mean degrees $k$.}
    \label{fig:Numerical-RGG-Supercritical-Fig-2}
\end{figure}

\emph{Model fit and stability improves with mean degree rather than size.} The test $R^2$ values are poor near the percolation threshold (at $k = 5$, $R^2 \approx 0.20$) but improve substantially as the mean degree increases, reaching $R^2\approx 0.8$ at $k=10$ and $R^2 \approx 0.96$ at $k = 20$ across all system sizes. Importantly, this improvement appears to be driven by increasing $k$ rather than $n$: for fixed $k$, the $R^2$ is nearly constant across $n \in \{512, \dots, 8192\}$. This suggests that the poor fit at low $k$ likely reflects the highly irregular tree-like structure of the graph near the percolation threshold, rather than finite-size effects. The LCC size data in Figure~\ref{fig:lcc} rules out the alternative explanation that the poor fit at $k = 5$ is due to the LCC being small: by $k = 6$ the LCC already contains the vast majority of nodes and the test $R^2$ for $k=6$ across system sizes hovers between $0.33$ and $0.37$ in our regressions. 

\begin{figure}[ht]
    \centering
    \includegraphics[width=0.95\textwidth]{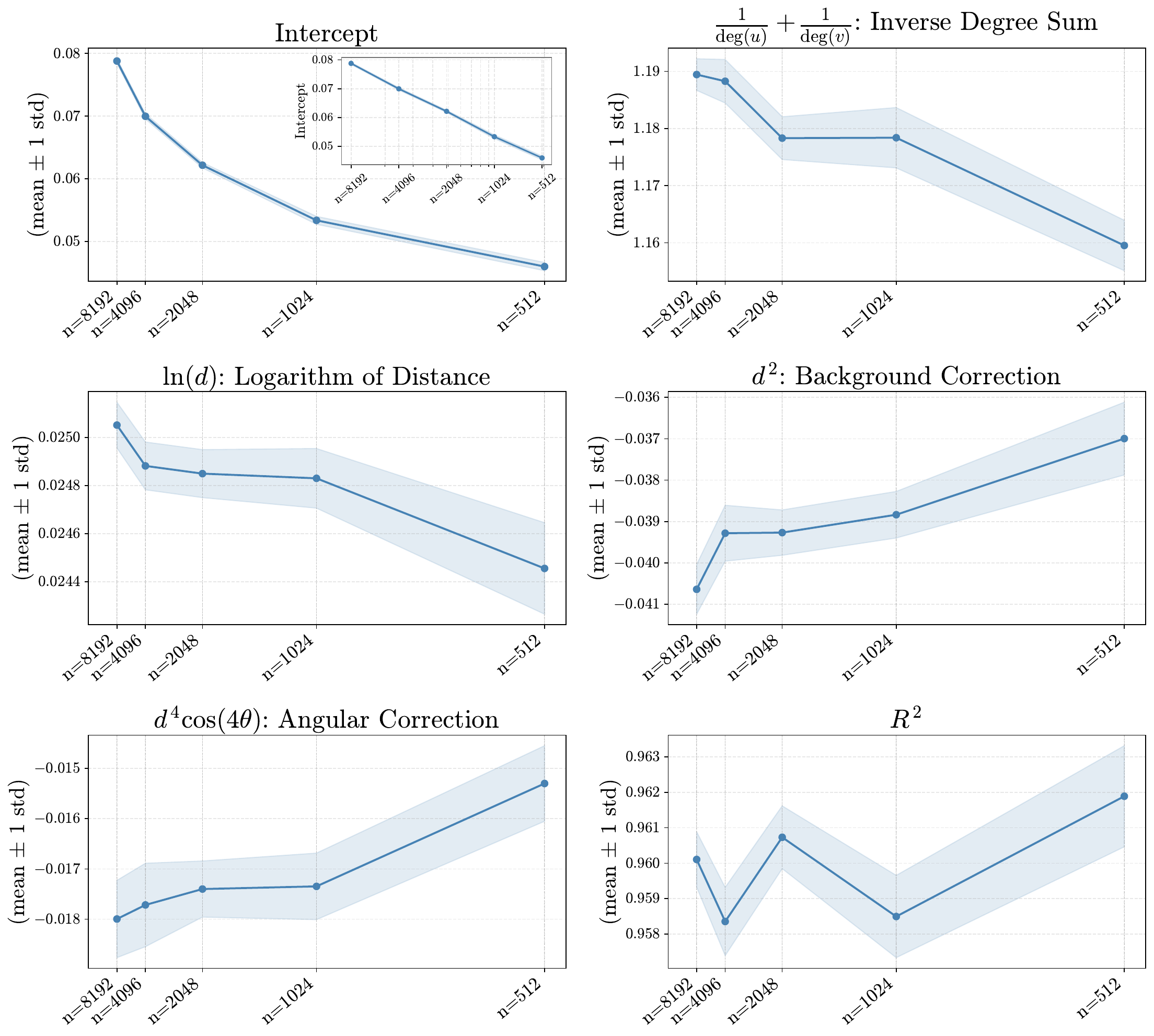}
    \caption{Visualization of Table~\ref{tab:rgg-supercritical} data plotting regression model fit and coefficients with respect to $1/n$ for mean degree $k=20$, emphasizing the increase in magnitude of the geometric coefficients with increasing $n$. Intercept inset confirms the expected linear dependence on $\ln(n)$.}
    \label{fig:Numerical-RGG-Supercritical-k=20-Coefs}
\end{figure}

\emph{Geometric terms remain significant even as their coefficients shrink.} As $k$ increases, the coefficients of the geometric terms $\ln(d)$, $d^2$, and $d^4\cos(4\theta)$ tend toward zero, linking to von Luxburg et al.'s degeneracy result: in the dense limit, the inverse degree sum dominates and the geometric structure of the graph becomes irrelevant. However, the absolute magnitude of these coefficients does not tell the whole story. As one example to assess their practical contribution to model fit at moderate $k$, we compare two models at $n = 1024$, $k = 20$: one using only the inverse degree sum $(D_u^{-1} + D_v^{-1})$, and one using all proposed features. The inverse-degree-only model achieves $R^2 = 0.7961 \pm 0.0044$, while adding the geometric terms (as in \cref{tab:rgg-supercritical}) raises this to $R^2 = 0.9585 \pm 0.0012$, a substantial improvement despite the geometric coefficients being small ($\beta_1 \doteq 0.025$, $\beta_2\doteq -0.039$, $\beta_3 \doteq -0.017$). That is, even when the geometric terms are small, their collective contribution to explaining variance in the effective resistance remains significant; the inverse degree sum alone is insufficient to fully characterize commute times even at $k$ this large. 

\emph{Coefficients are stable across shuffles.} The standard deviations of all coefficients across the $S$ train--test shuffles are small for $k \geq 8$, indicating that the model is stable with respect to individual graphs included in the training set. At $k = 5$, the geometric coefficients exhibit much larger relative variance, again likely reflecting high variability of individual graph realizations closer to the percolation threshold.

\section{The Transition from Square Lattices to RGGs} 
\label{sec:perturb}
To further assess the validity of the geometric terms in our approximation, we seek to study the transition from a fully regular geometry to the random node placement of an RGG. The (nearest-neighbor) square lattice studied in \cref{sec:model}, however, has regular degree $k=4$, below the percolation threshold; as such, perturbation of the nodes of this lattice will eventually break up the graph into small disconnected components. Meanwhile, the regression results of the preceding section show that our model achieves very good fit on RGGs with $k = 20$ across all system sizes tested. We therefore fix $k = 20$ throughout this section and connect nodes placed on a square lattice by using a fourth-order lattice in which each node is connected to all neighbors within a distance of $\sqrt{5}$ lattice spacings, yielding exactly $20$ neighbors per node.

\subsection{Perturbation Methodology} To study the transition from the regular lattice to a fully random geometric graph, we continuously deform the lattice by displacing each node by a random offset drawn independently from a bivariate normal distribution $\mathcal{N}(0, \varepsilon^2 I_2)$, where $\varepsilon$ is the perturbation scale and $I_2$ is the $2\times 2$ identity matrix. At $\varepsilon = 0$ the nodes are on a perfect square lattice (connected to their $20$ nearest neighbors). As $\varepsilon$ increases, the node positions become increasingly disordered, interpolating between the regular lattice and the uniform random placement of a standard RGG. We test perturbation values $\varepsilon \in \{0, 10^{-4}, 10^{-3}, 10^{-2}, 10^{-1}, 1, 10\}$, together with a fully random RGG at the same expected degree as a reference point. Recall the maximum distance on the unit torus is $d_{\max} = 1/\sqrt{2}\doteq 0.707$.

We set the connection radius $r$ so that the mean degree of the RGG satisfies $k=(n-1)\pi r^2\approx n\pi r^2=20$, that is, $r = 2\sqrt{5/n\pi}$. Importantly, recognizing $\sqrt{1/n}$ to be the spacing of the square lattice, this connection radius yields the same $k=20$ on that lattice; that is, the 4th nearest neighbors at $\sqrt{5}$ lattice spacings are within this connection radius but the 5th nearest neighbors at $2\sqrt{2}$ lattice spacings are not. As we will see in \cref{fig:lattice-perturbation}, the mean degrees of perturbed lattices with this constant connection radius differ only slightly and with small variance.

\subsection{Numerical Results}

\cref{fig:lattice-perturbation} shows the changes in the regression coefficients and model fit as the perturbation  increases from $\varepsilon = 0$ (perfect lattice) to $\varepsilon = 10$ (heavily disordered), with the standard RGG ($\varepsilon\to\infty$) as a reference.

\begin{figure}[ht]
    \centering
    \includegraphics[width=0.95\textwidth]{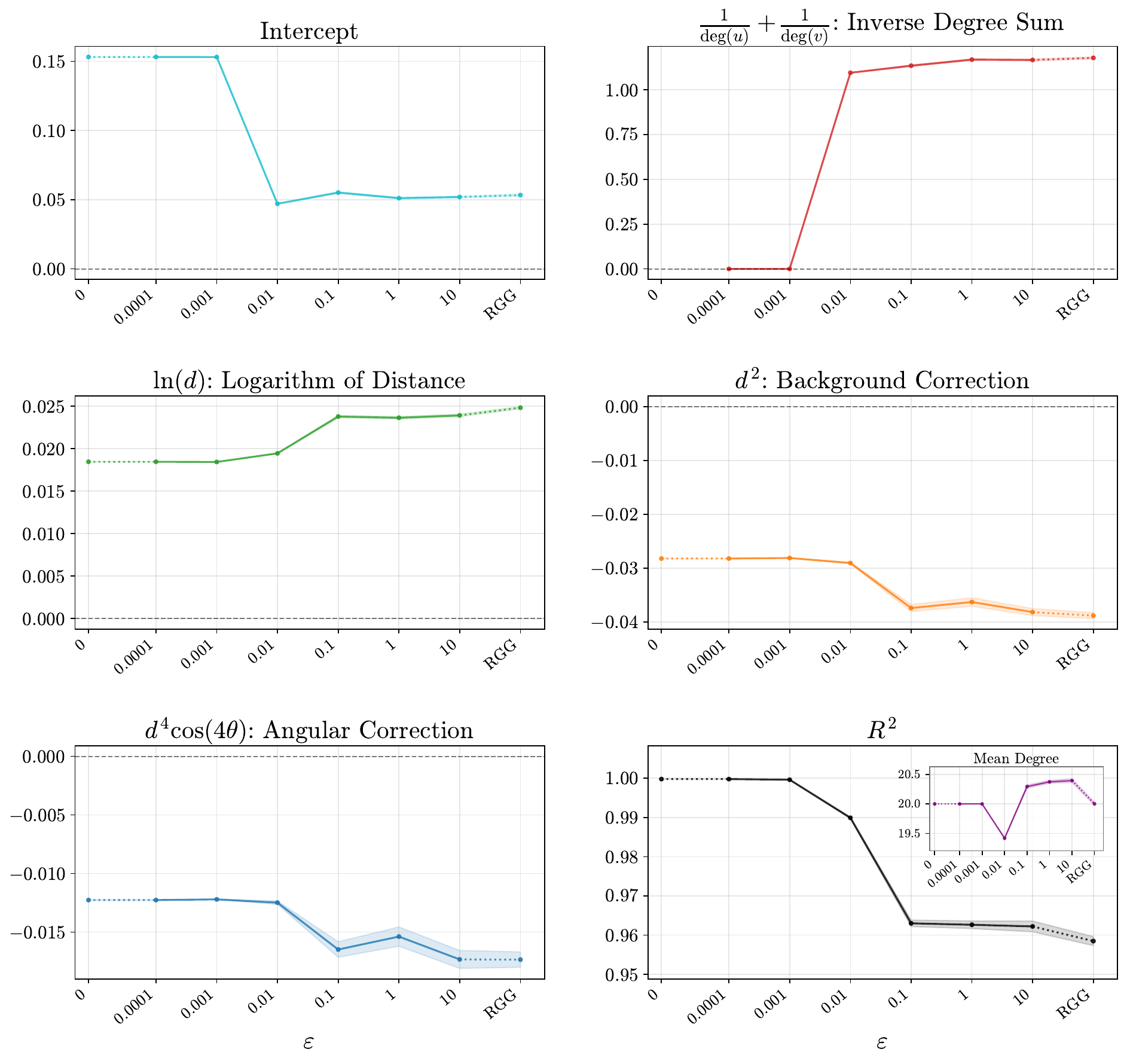}
    \caption{Evolution of regression coefficients, test $R^2$ fit, and mean degree versus increasing lattice perturbation $\varepsilon$ for $n = 1024$ nodes at connection radius $r = 2\sqrt{5/n\pi}$ (i.e., $k \approx 20$). Each point is the mean over 50 independent train--test shuffles; shaded bands show $\pm 1$ standard deviation. The rightmost points labeled ``RGG'' are random geometric graphs at $k=20$, included for reference.}
    \label{fig:lattice-perturbation}
\end{figure}

Note that the inverse degree sum coefficient obtained by regression is zero by construction at $\varepsilon = 0$, since all nodes have identical degree on the perfect lattice so this feature carries no information. Degree heterogeneity emerges once a large enough perturbation is introduced ($\varepsilon \geq 10^{-2}$), at which point the inverse degree sum immediately becomes the largest coefficient, rising sharply to approximately $1.1$ by $\varepsilon = 10^{-2}$ and stabilizing to $\approx 1.18$ thereafter. This rapid onset confirms the importance of degree heterogeneity for effective resistance, even for small geometric disorder. We additionally note that this sharp appearance of the nonzero inverse degree sum coefficient is paired with a corresponding decrease in the model intercept from 0.15 to 0.05, a change in accordance with the $k=20$ homogeneous (unperturbed) inverse degree sum ($\frac{1}{20}+\frac{1}{20}=0.1$).

The geometric coefficients $\ln(d)$, $d^2$, and $d^4\cos(4\theta)$ behave relatively similarly across the transition from unperturbed lattice to RGG. The logarithmic coefficient increases gradually from the lattice value before stabilizing near the RGG value, with an overall increase of $34.5\%$. The quadratic coefficient $d^2$ and the angular coefficient $d^4\cos(4\theta)$ likewise increase in magnitude overall as $\varepsilon$ increases, by approximately $37.6\%$ and $41.4\%$ respectively.

Meanwhile, the $R^2$ values inform us that the model fit degrades slightly with increasing perturbation but nevertheless fits well throughout the transition. The model fits the unperturbed lattice very well and remains above $0.9996$ for small perturbations (e.g. $\varepsilon \leq 10^{-3}$) where the degrees remain regular. The $R^2$ values drop through $\varepsilon = 10^{-2}$ and $10^{-1}$ as the geometric disorder becomes large enough that the lattice-derived features no longer fully capture the resistances, before stabilizing near the RGG value of $R^2 \approx 0.96$. The mean degree for this fixed connection radius remains relatively stable, with the largest exception being the drop observed at $\varepsilon = 10^{-2}$ before recovering to slightly larger than $20$ for larger perturbations.

Taken together, these results suggest that the functional form of our model is well-motivated across the full transition from regular lattice to random geometric graph: the same terms --- $\ln(d)$, $d^2$, $d^4\cos(4\theta)$, the inverse degree sum, and an intercept --- provide good predictive power throughout, with the relative importance of each term shifting smoothly as the geometric disorder increases.

Notably, whereas one might expect the effect of the geometric features to weaken as the regular structure of the lattice is destroyed, with the inverse degree sum progressively absorbing a greater role in determining the distances, we instead find that the coefficients of the geometric features increase in magnitude as the perturbation of the lattice increases. (We do also note that the coefficient of the inverse degree sum increases slightly with increasing perturbation.) We hypothesize that this behavior may be a consequence of the mean degree $k = 20$ used in these experiments and its limitation of a fit with $R^2=0.96$ for the RGG. That is, the slight decrease in $R^2$ observed across the transition suggests that the model is under some strain and that the geometric coefficients are compensating for variance that might be more stable in a denser graph or may benefit from higher-order terms. Finite-size effects at $n = 1024$ may also play a role. The dependence on mean degree, inclusion of additional terms, and finite-size effects are each interesting directions for future work.

\subsection{Remarks on the Lattice Perturbation Experiment} 
The experiments in this section establish three main findings. First, the four-feature model --- $\ln(d)$, $d^2$, $d^4\cos(4\theta)$, and the inverse degree sum (plus an intercept) --- provides good predictive power across the full range of graph structures from the regular square lattice to the fully random RGG, for graphs with expected mean degree 20, with $R^2$ remaining above $0.96$ throughout. Second, the coefficients vary only slightly throughout the perturbation, with the expected exception of the sudden appearance of the inverse degree sum coefficient (and corresponding change in the intercept) once the perturbation becomes large enough to generate degree heterogeneity, demonstrating that even a small amount of degree heterogeneity is sufficient to make local degree information the primary determinant of effective resistance. Third, the geometric terms $\ln(d)$, $d^2$, and $d^4\cos(4\theta)$ remain present and statistically significant throughout the transition, including in the fully random RGG regime, confirming that the global geometry of the graph, as inherited from the spatial domain, retains explanatory power beyond what local degree information alone can provide. Together, these findings validate the functional form of our model and motivate its application to sparse RGGs in the supercritical regime. In cases where degree heterogeneity is more extreme, the inverse degree sum is expected to play an even more prominent role, and the geometric terms may behave quite differently from what is observed  numerically here.

\section{Conclusion}
\label{sec:conclusions}

In this study we investigated the relationship between commute times and the underlying Euclidean metric for sparse random geometric graphs in regimes where the existing analytical theory is incomplete. Our central contribution is a four-feature regression model for commute times in RGGs on the torus, combining geometric terms --- $\ln(d)$, $d^2$, and $d^4\cos(4\theta)$ --- with the inverse degree sum that dominates in the limit of sufficiently dense RGGs described by von Luxburg et al.'s \cite{vonLuxburg2014} theorem. We validated our model numerically across a range of graph sizes and mean degrees, demonstrated its stability under train--test shuffles, and traced its continuous evolution from a regular square lattice to fully random RGG through a controlled perturbation experiment. Importantly, these results indicate that the two-dimensional geometry of the embedding space continues to play an important role in setting commute times even at moderate mean degrees. We performed all numerical calculations with our lightweight Python library \cite{rupchin2025rgg}, which we make available to allow our results to be easily reproduced, verified, and extended.

Our study has focused on the case with constant expected mean degree in the supercritical regime. The sharply fully connected regime --- where mean degree scales with $\ln n$, that is,
$(n-1)\pi r_n^2 \sim \ln n$, at which the graphs are just barely fully connected --- is equally important from both a theoretical and practical standpoint, and is not dense enough in the limit to satisfy the conditions of von Luxburg et al.'s theorem, but is significantly more challenging to study numerically. We hope that the model developed here will serve as a starting point for future investigation of this regime.

We also hope that the empirical relationships uncovered here will serve as a guide for future rigorous analytical work. We anticipate that attempts to obtain sharp asymptotic results for the supercritical regime would include the geometric structural features that appear in our numerical results. Our findings suggest that the coefficients of the geometric terms converge to well-defined limiting values as $n$ tends to infinity at fixed mean degrees, which would be consistent with the existence of a sharp asymptotic formula analogous to the dense-regime von Luxburg et al.\ result. We emphasize that, while the $\ln(d)$ and $d^2$ terms are isotropic, the persistent appearance of the $d^4\cos(4\theta)$ term across our results signals a dependence on the geometry of the underlying torus within which we performed all of our numerical experiments. The higher-order nature of this square-symmetric term conveniently decreases its impact for $d\ll 1$ --- that is, at distances short enough to be ``far'' from the boundary condition effects of the toroidal geometry. We expect analogous factors could be important for trying to understand the possible geometric effects in other embedding spaces. 

There are also several natural directions for future study in terms of possible applications, where the regression model developed here might serve as a basis for practical heuristics for approximating commute times and effective resistances in large spatially-embedded networks, avoiding the computational cost of full Laplacian inversion. For example, such heuristics might be useful in network routing, community detection, and distributed estimation problems \cite{ghosh2008minimizing, yen2007graph}, where exact computation is infeasible at scale but a fast geometric approximation might be sufficient. The class of graphs studied could also be extended to more involved network models with spatial structure, such as soft RGGs \cite{penrose2016connectivity,zhukovskii2010rdg}, or RGGs on the (not periodic) unit square where boundary effects will introduce additional complexity in part because vertices near the boundary have systematically lower degree. In general, this study contributes to the larger program of understanding when and how the geometry of a spatial embedding relates to the dynamics of processes on networks; as such, we are eager to see similar relationships identified in other related network settings.

\appendix

\bibliographystyle{siamplain}
\bibliography{references}
\end{document}

%% file: ex_shared.tex
\usepackage{lipsum}
\usepackage{amsfonts}
\usepackage{graphicx}
\usepackage{epstopdf}
\usepackage{algorithmic}
\ifpdf
  \DeclareGraphicsExtensions{.eps,.pdf,.png,.jpg}
\else
  \DeclareGraphicsExtensions{.eps}
\fi

\newsiamremark{remark}{Remark}
\newsiamremark{hypothesis}{Hypothesis}
\crefname{hypothesis}{Hypothesis}{Hypotheses}
\newsiamthm{claim}{Claim}
\newsiamremark{fact}{Fact}
\crefname{fact}{Fact}{Facts}

\headers{Commute Times in Sparse Random Geometric Graphs}{Filip O. Rupchin \& Peter J. Mucha}

\title{Effective Resistances and Commute Times\\ in Sparse Random Geometric Graphs\thanks{Submitted to the editors June 12, 2026. The authors acknowledge the use of Anthropic's Claude and Google's Gemini to help draft the literature review, format citations, and write parts of the codes for our numerical experiments and figures. The authors assume responsibility for all content.
\funding{This work was funded by the URAD program at Dartmouth College and the Jack Byrne Academic Cluster in Mathematics and Decision Science at Dartmouth College.}}}

\author{
Filip O. Rupchin\thanks{Dartmouth College, Hanover, NH 
  (\email{frupchin@gmail.com}).}
  \and
Peter J. Mucha\thanks{Department of Mathematics, Dartmouth College, Hanover, NH 
  (\email{peter.j.mucha@dartmouth.edu}).}
}

\usepackage{amsopn}


%% file: references.bib
@article{Tetali_1991, title={Random walks and the effective resistance of networks}, volume={4}, ISSN={0894-9840, 1572-9230}, DOI={10.1007/BF01046996}, abstractNote={In this article we present an interpretation ofeffective resistance in electrical networks in terms of random walks on underlying graphs. Using this characterization we provide simple and elegant proofs for some known results in random walks and electrical networks. We also interpret the Reciprocity theorem of electrical networks in terms of traversals in random walks. The byproducts are (a) precise version of thetriangle inequality for effective resistances, and (b) an exact formula for the expectedone-way transit time between vertices.}, number={1}, journal={Journal of Theoretical Probability}, author={Tetali, Prasad}, year={1991}, month=jan, pages={101–109}, language={en} }

@article{Coppersmith_Tetali_Winkler_1993, title={Collisions Among Random Walks on a Graph}, volume={6}, ISSN={0895-4801}, number={3}, journal={SIAM Journal on Discrete Mathematics}, publisher={Society for Industrial and Applied Mathematics}, author={Coppersmith, Don and Tetali, Prasad and Winkler, Peter}, year={1993}, month=aug, pages={363–374} }

@article{zhukovskii2010rdg,
  author  = {Zhukovskii, Maksim E.},
  title   = {The Weak Zero-One Laws for the Random Distance Graphs},
  journal = {Doklady Mathematics},
  volume  = {84},
  number  = {1},
  pages   = {51--54},
  year    = {2010}
}

@misc{sidford2018lecture13,
  author       = {Sidford, Aaron},
  title        = {Lecture 13 - {R}andom {W}alks},
  howpublished = {CME 305: Discrete Mathematics and Algorithms, Stanford University},
  year         = {2018},
  url          = {https://web.stanford.edu/class/cme305/Files/l13.pdf}
}

@book{chung1997spectral,
  author    = {Chung, Fan R. K.},
  title     = {Spectral Graph Theory},
  publisher = {American Mathematical Society},
  year      = {1997},
  series    = {CBMS Regional Conference Series in Mathematics},
  volume    = {92}
}

@article{klein1993resistance,
  author  = {Klein, Douglas J. and Randi{\'c}, Milan},
  title   = {Resistance Distance},
  journal = {Journal of Mathematical Chemistry},
  volume  = {12},
  pages   = {81--95},
  year    = {1993}
}

@misc{cnbc2025instagram,
  author       = {{CNBC}},
  title        = {Instagram Now Has 3 Billion Monthly Active Users},
  howpublished = {\url{https://www.cnbc.com/2025/09/24/instagram-now-has-3-billion-monthly-active-users.html}},
  year         = {2025},
  month        = sep,
  day          = {24},
  note         = {Accessed: 2026-03-17}
}

@article{chandra1996electrical,
  author  = {Chandra, Ashok K. and Raghavan, Prabhakar and Ruzzo, Walter L. 
             and Smolensky, Roman and Tiwari, Prasoon},
  title   = {The Electrical Resistance of a Graph Powers and Superpowers},
  journal = {SIAM Journal on Computing},
  volume  = {25},
  number  = {1},
  pages   = {235--246},
  year    = {1996}
}

@book{doyle1984random,
  author    = {Doyle, Peter G. and Snell, J. Laurie},
  title     = {Random Walks and Electric Networks},
  publisher = {Mathematical Association of America},
  year      = {1984},
  note      = {Available at \url{https://arxiv.org/abs/math/0001057}}
}

@misc{spielman2019spectral,
  author       = {Spielman, Daniel A.},
  title        = {Spectral Graph Theory},
  howpublished = {Lecture notes, Yale University},
  year         = {2019},
  url          = {http://www.cs.yale.edu/homes/spielman/sgt/}
}

@inproceedings{yen2007graph,
  author    = {Yen, Luh and Fouss, Fran{\c{c}}ois and Decaestecker, Christine and Francq, Pascal and Saerens, Marco},
  title     = {Graph Nodes Clustering Based on the Commute-Time Kernel},
  booktitle = {Advances in Knowledge Discovery and Data Mining (PAKDD)},
  series    = {Lecture Notes in Computer Science},
  volume    = {4426},
  pages     = {1037--1045},
  year      = {2007},
  publisher = {Springer}
}

@inproceedings{spielman2004nearly,
  author    = {Spielman, Daniel A. and Teng, Shang-Hua},
  title     = {Nearly-Linear Time Algorithms for Graph Partitioning, Graph Sparsification, and Solving Linear Systems},
  booktitle = {Proceedings of the 36th Annual ACM Symposium on Theory of Computing (STOC)},
  pages     = {81--90},
  year      = {2004},
  publisher = {ACM}
}

@inproceedings{fouss2004novel,
  author    = {Fouss, Fran{\c{c}}ois and Yen, Luh and Pirotte, Alain and Saerens, Marco},
  title     = {An Experimental Investigation of Graph Kernels on a Collaborative Recommendation Task},
  booktitle = {Proceedings of the 6th International Conference on Data Mining (ICDM)},
  year      = {2006},
  publisher = {IEEE}
}

@article{fouss2007random,
  author  = {Fouss, Fran{\c{c}}ois and Pirotte, Alain and Renders, Jean-Michel and Saerens, Marco},
  title   = {Random-Walk Computation of Similarities Between Nodes of a Graph with Application to Collaborative Recommendation},
  journal = {IEEE Transactions on Knowledge and Data Engineering},
  volume  = {19},
  number  = {3},
  pages   = {355--369},
  year    = {2007}
}

@article{ghosh2008minimizing,
  author  = {Ghosh, Arpita and Boyd, Stephen and Saberi, Amin},
  title   = {Minimizing Effective Resistance of a Graph},
  journal = {SIAM Review},
  volume  = {50},
  number  = {1},
  pages   = {37--66},
  year    = {2008}
}

@article{coifman2006diffusion,
  author  = {Coifman, Ronald R. and Lafon, St{\'e}phane},
  title   = {Diffusion Maps},
  journal = {Applied and Computational Harmonic Analysis},
  volume  = {21},
  number  = {1},
  pages   = {5--30},
  year    = {2006}
}

@book{penrose2003random,
  author    = {Penrose, Mathew},
  title     = {Random Geometric Graphs},
  publisher = {Oxford University Press},
  year      = {2003}
}

@article{erdos1959random,
  author  = {Erd{\H{o}}s, Paul and R{\'e}nyi, Alfr{\'e}d},
  title   = {On Random Graphs},
  journal = {Publicationes Mathematicae Debrecen},
  volume  = {6},
  pages   = {290--297},
  year    = {1959}
}

@book{bollobas2001random,
  author    = {Bollob{\'a}s, B{\'e}la},
  title     = {Random Graphs},
  edition   = {second},
  publisher = {Cambridge University Press},
  year      = {2001}
}

@article{barabasi1999emergence,
  author  = {Barab{\'a}si, Albert-L{\'a}szl{\'o} and Albert, R{\'e}ka},
  title   = {Emergence of Scaling in Random Networks},
  journal = {Science},
  volume  = {286},
  number  = {5439},
  pages   = {509--512},
  year    = {1999}
}

@article{holland1983stochastic,
  author  = {Holland, Paul W. and Laskey, Kathryn Blackmond and Leinhardt, Samuel},
  title   = {Stochastic Blockmodels: First Steps},
  journal = {Social Networks},
  volume  = {5},
  number  = {2},
  pages   = {109--137},
  year    = {1983}
}

@article{gilbert1961random,
  author  = {Gilbert, Edgar N.},
  title   = {Random Plane Networks},
  journal = {Journal of the Society for Industrial and Applied Mathematics},
  volume  = {9},
  number  = {4},
  pages   = {533--543},
  year    = {1961}
}

@article{penrose2016connectivity,
  author  = {Penrose, Mathew D.},
  title   = {Connectivity of Soft Random Geometric Graphs},
  journal = {Annals of Applied Probability},
  volume  = {26},
  number  = {2},
  pages   = {986--1028},
  year    = {2016}
}

@misc{penrose2022largec,
  author       = {Penrose, Mathew D.},
  title        = {Large Components of Random Geometric Graphs},
  howpublished = {Lecture notes, Mini-course: Munich ``Summer'' School 2022 on Discrete Random Systems},
  year         = {2022},
  url          = {https://www.mathematik.uni-muenchen.de/~heyden/Notes.Penrose.pdf}
}

@article{Quintanilla2000,
  author  = {Quintanilla, John and Torquato, Salvatore and Ziff, Robert M.},
  title   = {Efficient measurement of the percolation threshold for fully penetrable discs},
  journal = {Journal of Physics A: Mathematical and General},
  volume  = {33},
  number  = {42},
  pages   = {L399--L407},
  year    = {2000}
}

@article{vonLuxburg2014,
  author  = {von Luxburg, Ulrike and Radl, Agnes and Hein, Matthias},
  title   = {Hitting and Commute Times in Large Random Neighborhood Graphs},
  journal = {Journal of Machine Learning Research},
  volume  = {15},
  number  = {1},
  pages   = {1751--1798},
  year    = {2014}
}

@article{cserti2000,
  author  = {Cserti, J{\'o}zsef},
  title   = {Application of the lattice {G}reen's function for calculating 
             the resistance of infinite networks of resistors},
  journal = {American Journal of Physics},
  volume  = {68},
  number  = {10},
  pages   = {896--906},
  year    = {2000}
}

@article{jeng2004,
  author  = {Jeng, Monwhea},
  title   = {Random walks and effective resistances on toroidal and 
             cylindrical grids},
  journal = {American Journal of Physics},
  volume  = {68},
  number  = {1},
  pages   = {37--40},
  year    = {2000}
}

@article{rossi2015,
  author  = {Rossi, Wilbert Samuel and Frasca, Paolo and Fagnani, Fabio},
  title   = {Average resistance of toroidal graphs},
  journal = {SIAM Journal on Control and Optimization},
  volume  = {53},
  number  = {4},
  pages   = {2541--2560},
  year    = {2015}
}

@misc{rupchin2025rgg,
  author       = {Rupchin, Filip},
  title        = {Random Geometric Graphs and Effective Resistance},
  year         = {2025},
  url = {https://github.com/FilipRupchin/Random-Geometric-Graphs-and-Effective-Resistance}
}

@inproceedings{radl2009, title={The resistance distance is meaningless for large random geometric graphs}, booktitle={Proc. Workshop on Analyzing Networks and Learning with Graphs}, author={Radl, Agnes and Von Luxburg, Ulrike and Hein, Matthias}, year={2009} }

@article{bonchi2012, title={Fast Matrix Computations for Pairwise and Columnwise Commute Times and Katz Scores}, volume={8}, ISSN={1542-7951, 1944-9488}, DOI={10.1080/15427951.2012.625256}, number={1–2}, journal={Internet Mathematics}, author={Bonchi, Francesco and Esfandiar, Pooya and Gleich, David F. and Greif, Chen and Lakshmanan, Laks V.S.}, year={2012}, month=mar, pages={73–112} }

@article{McPherson1983, title={An Ecology of Affiliation}, volume={48}, ISSN={0003-1224}, DOI={10.2307/2117719}, number={4}, journal={American Sociological Review}, publisher={[American Sociological Association, Sage Publications, Inc.]}, author={McPherson, Miller}, year={1983}, pages={519–532} }
